\documentclass[11pt]{article}
\usepackage[utf8]{inputenc}
\usepackage{amssymb,amsfonts,amsthm}
\usepackage{amsthm,mathtools}
\usepackage{amssymb}
\usepackage[francais,english]{babel}
\usepackage{setspace}
\usepackage{bm}
\usepackage{imakeidx}
\usepackage{makeidx}
\usepackage{verbatim}
\usepackage[noblocks]{authblk}
\usepackage[dvipsnames]{color}

\usepackage{amsfonts,amsmath,layout}

\newcommand{\Z}{\mathbb Z}

\newcommand{\R}{\mathbb R}

\newcommand{\T}{\mathbb{T}}
\newcommand{\mPT}{\mathcal{P}(\T^d)}

\def\R{\mathbb R}

\def\Z{\mathbb Z}

\newcommand{\be}{\begin{equation}}
\newcommand{\ee}{\end{equation}}
\newcommand{\ba}{\begin{align}}
\newcommand{\ea}{\end{align}}

\def\1{{\bf 1}}

\def\inte{\int_{\T^d}}
\def\dive{{\rm div}}

\newcommand*\Laplace{\mathop{}\!\mathbin\bigtriangleup}

\newtheorem{teo}{Theorem}[section]
\newtheorem{defin}[teo]{Definition}
\newtheorem{prop}[teo]{Proposition}
\newtheorem{lemma}[teo]{Lemma}

\newtheorem{remark}[teo]{Remark}

\mathtoolsset{showonlyrefs}

\usepackage[dvipsnames]{xcolor}
\definecolor{mygray}{gray}{0.7}
\usepackage[pdftex]{graphicx}
\definecolor{Red}{cmyk}{0,1,1,0.2}

\usepackage[colorinlistoftodos,prependcaption,textsize=normal]{todonotes}

\begin{document}
	\title{Convergence of the solutions of the MFG discounted Hamilton-Jacobi equation}
	\author{Marco Masoero \thanks{PSL Research University, Universit\'e Paris-Dauphine, CEREMADE, Place de Lattre de Tassigny, F- 75016 Paris, France}}
	
	\maketitle
	
    
    \begin{abstract}
    	We consider the solution $\mathcal V_\delta$ of the discounted Hamilton-Jacobi equation in the Wasserstein space arising from potential MFG and we prove its full convergence to a corrector function $\chi_0$. We follow the structure of the proof of the analogue result in the finite dimensional setting provided by Davini, Fathi, Iturriaga, Zavidovique in 2017.  We characterize the limit $\chi_0$ through a particular set of smooth Mather measures. A major point that distinguishes the techniques deployed in the standard setting from the ones that we use here is the lack of mollification in the Wasserstain space.
    \end{abstract}

\section{Introduction}
The Mean Field Games theory (briefly MFG) is a branch of the broader field of dynamic games which is devoted to the analysis of those models where a large number of \textit{small} players interact with each others. This theory was introduced simultaneously and independently by Lasry and Lions \cite{lasry2006jeux,lasry2006jeux2} and Huang, Caines and Malhamé \cite{huang2006large}. Under appropriate assumptions, the Nash equilibria of this models can be analyzed through the solutions of the, so called, MFG system 
$$
\begin{cases} 
-\partial_t u-\Laplace u+H(x,Du)=F(x,m) & \mbox{in } \R^d\times[0,T]\\
-\partial_t m+\Laplace m+{\rm div}(mD_pH(x,Du))=0 & \mbox{in }\R^d\times[0,T]\\
m(0)=m_0,\; u(T,x)=G(x,m(T))&\mbox{in }\R^{d}
\end{cases}
$$
with unknown the couple $(u,m)$. The value $u(t,x)$ is the best that a player can get starting from $x$ at time $t$ while $m(t)$ is a probability measure that represents the distribution of players at time $t$.

In this paper we will consider a specific class of MFG which is the class of \textit{potential MFG}. When the functions $F$ and $G$ are respectively the derivatives of the potentials $\mathcal F$ and $\mathcal G$, the MFG system can be derived as optimality condition of the following minimization problem
\begin{equation}\label{intro.Ut.delta}
\mathcal U^T(t,m_{0})= \inf_{(m,\alpha)}\int_t^T\int_{\R^d}H^*\left(x,\alpha(s,x)\right)dm(s)+\mathcal F(m(s))ds+\mathcal G(m(T)),
\end{equation}
where $(m,\alpha)$ verifies the Fokker Plank equation $-\partial_{t}m+\Laplace m+{\rm div}(\alpha m)=0$ with $m(t)=m_{0}$.

Starting from \cite{lasry2006jeux2}, where this class of MFG was introduced, several papers have been focusing on this setting. The variational structure of these models often allows to push the analysis further than in the standard setting. See for instance \cite{cardaliaguet2015second,meszaros2018variational,cardaliaguet2015mean} for existence results or \cite{LavSan18,ProSan16,cardaliaguet2016first} for regularity.
 An other reason to exploit the variational structure of potential MFG is to understand how the solutions of the MFG system behave when the time horizon goes to infinity. The problem of the long time convergence has been addressed in different papers starting from \cite{lionsmean} and the Mexican wave model in \cite{gueant2011mean} to more recent contributions in \cite{cardaliaguet2013long,cardaliaguet2013long2,cardaliaguet2012long,gomes2010discrete}. In \cite{masoero2019,cardaliaguet2019weak} the authors tackled the problem using techniques from weak KAM theory. Following Fathi's seminal papers \cite{fathi1997solutions,fathi1997theoreme,fathi1998convergence} and his book \cite{fathi2008weak}, they adapted the main arguments of the weak KAM theory into the infinite dimensional framework of MFG. Note that, infinite dimensional weak KAM theorems were already known in the context of Wasserstein spaces (see \cite{gamgbo2010lagrangian,gamgbo2014weak,gomes2015minimizers,gomes2016infinite}). 

This papers is meant to answer a natural question that arises in this theory and more specifically in the context of \textit{ergodic approximation}. In \cite{cardaliaguet2019weak}, it was proved that there exists a critical value $\lambda\in\R$ such that $\mathcal U^T(0,\cdot)+\lambda T$ uniformly converges to a corrector function $\chi$ when $T\rightarrow+\infty$. We say that $\chi:\mathcal P(\T^d)\rightarrow\R$ is a corrector function if, for any $t\in\R$, $\chi$ verifies the following dynamical programming principle
\begin{equation}\label{intro.eq.defchichi.delta}
\chi(m_{0})= \inf_{(m,\alpha)}\left( \int_{0}^{t}H^{*}(x,\alpha)dm(s)+\mathcal F(m(s))ds+\chi (m(t))\right)+\lambda t,
\end{equation}
where $(m,\alpha)$ solves in the sense of distributions $-\partial_{t}m+\Delta m+{\rm div}(m\alpha)=0$ with initial condition $m(0)=m_{0}$. 
A fundamental result that was needed to get the above convergence was to know a priori that the set of corrector functions was not empty. In \cite{masoero2019} this was proven through the so called ergodic approximation.  As in \cite{lions1987homogenization}, the idea is to define the infinite horizon discounted problem 

$$
\mathcal V_{\delta}(m_{0})=\inf_{(m,\alpha)}\int_0^\infty e^{-\delta t}\int_{\T^d}H^*\left(x,\alpha(t,x)\right)dm(t)+\mathcal F(m(t))dt.
$$
 Letting $\delta\rightarrow 0$, one gets that $\delta \mathcal V_\delta(\cdot)\rightarrow-\lambda$ and that, up to subsequence, $\mathcal V_\delta(\cdot)-V_\delta(m_0)$ uniformly converges to a corrector function $\chi$.

It was proved by Davini, Fathi, Iturriaga and Zavidovique \cite{davini2016convergence} that, in the standard finite dimensional setting, the solution of the discounted equation converges to a solution of the critical Hamilton-Jacobi equation, i.e. to a corrector function. In this paper we prove the analogous result in the context of potential MFG. If we define 
\begin{align}\label{}
\bar{\mathcal V}_\delta (m)=\mathcal V_\delta(m)+\frac{\lambda}{\delta},
\end{align}
then the main result is the following. The whole family $\bar{\mathcal V}_\delta$ converges, as $\delta$ tends to zero, to a corrector function $\chi_0$. Moreover, we can characterize this limit as
\begin{align}\label{}
\chi_0(m)=\sup_{\chi\in{\mathcal S}^-}\chi(m).
\end{align}
where $\mathcal S^-$ is the set of subsolutions $\chi$ of 
\begin{equation}\label{intro.critical.HJ}
-\int_{\T^d}\dive_y D_m\chi(m,y) m(dy)+\int_{\T^d} H(y, D_m\chi(m,y))m(dy)-\mathcal F(m)=\lambda,\qquad m\in\mathcal P(\T^d)
\end{equation}
such that $\int_{{\mathcal P}(\T^d)}\chi(m)\nu(dm)$ for any $\nu\in\mathcal M_{\mathcal V}$. We say that a $\nu\in\mathcal M_{\mathcal V}$ if it is induced by an optimal trajectories for $\mathcal V_\delta$: if $(m_\delta,\alpha_\delta)$ is an optimal trajectory for $\bar{\mathcal V}_\delta(m_0)$ we define $\nu_\delta^{m_0}\in\mathcal P(\T^d\times C^1(\T^d,\R^d))$ as follows
\begin{align}
\int_{{\mathcal P}(\T^d)\times C^1(\T^d,\R^d)}f(m,\alpha)\nu_\delta^{m_0}(dm,d\alpha)=\delta\int_0^\infty e^{-\delta s}f(m^\delta(s),\alpha^\delta(s))ds.
\end{align}
If $\nu^{m_0}$ is the limit of $\nu_\delta^{m_0}$ and $m_0$ has smooth density, then $\nu^{m_0}\in\mathcal M_{\mathcal V}$. Moreover, we will also prove that if $\nu^{m_0}\in\mathcal M_{\mathcal V}$ then $\nu^{m_0}$ is a \textit{Mather measure} (Definition \ref{def.mather.measure}).

The structure of the paper itself is inspired by the one in \cite{davini2016convergence}. Even though the steps to get to the result are mostly the same, the techniques deployed to prove the main points are quite different. In \cite{davini2016convergence}, a major ingredient is that one can approximate a viscosity solution with a smooth function and this approximation is also an approximated solution of the equation. While this is quite standard in the finite dimensional setting and it is generally proved through mollification, in the space of functions over $\mathcal P(\T^d)$ one cannot expect the same. In \cite{mou2019weak} it was recently proved that one can uniformly approximate a continuous function on $\mathcal P(\T^d)$ with a sequence of smooth ones (similar approximation were also introduced in \cite{lionscollege}). The problem is that this convergence is not as strong as the one given by mollification in the finite dimensional setting. More specifically, we know that if $\chi$ is a corrector function then it is also a viscosity solution of the critical equation \eqref{intro.critical.HJ}. The result in \cite{mou2019weak} allows us to approximate $\chi$ uniformly but one cannot expect that this approximation is also an approximated solution of \eqref{intro.critical.HJ}. This lack of regularity prevented us to work at the level of solutions of the critical equation and it forced us to rely only on the dynamical programming principle \eqref{intro.eq.defchichi.delta} and the properties of optimal trajectories. Moreover, the lack of regularity led to an other difference. The characterization of the limit $\chi_0$ used by Davini et al. is slightly different. In their definition of $\mathcal S^-$ the subsolutions are tested against any Mather measure and not only against the subset $\mathcal M_{\mathcal V}$.

We now briefly discuss the structure of the paper. In Section \ref{sec.defin.delta} we set the problem in term of assumptions and notation and we collect the main results previously proved in \cite{masoero2019,cardaliaguet2019weak} that are used in the paper.
 
  Section \ref{sec.conv.Vdelta} is devoted to the convergence of $\bar{\mathcal V}_\delta$. In Subsection \ref{subsec.meas.ind} we introduce the notion of probability measures induced by optimal trajectories of $\mathcal V_\delta$ and we prove that the limit of these probability measures are Mather measures.
  
  In Subsection \ref{subsec.lower.bound} we prove that, if $\chi$ is a subsolution of \eqref{intro.critical.HJ} and $\nu_\delta^{m_0}$ is a probability measure induced by the optimal trajectory $(m_\delta,\alpha_\delta)$, then
  	$$
  \bar{\mathcal{V}}_\delta(m_0)\geq\chi(m_0)-\int_{\mathcal P(\T^d)}\chi(m)\nu_\delta^{m_0}(dm).
  $$
  The proof is done comparing the dynamic programming principle verified by $\bar{\mathcal V}_\delta$ and the one verified by the subsolution $\chi$. 
  
  In Subsection \ref{subsec.smooth.mather} we introduce the notion of \textit{smooth Mather measure} (Definition \ref{defin.smooth.measure}) which are Mather measures which have smooth densities. The smoothness of these measures allows to overcome the lack of regularity of $\bar{\mathcal V}_\delta$ so that we can prove that if $(\mu,p_1)$ is a smooth Mather measure then
  \begin{align}\label{}
  \delta\int_{{\mathcal P}(\T^d)}\bar{\mathcal V}_\delta(\bar m)\mu(d\bar m)\leq 0.
  \end{align}
  
  Moreover, we prove that if $\nu^{m_0}\in\mathcal M_{\mathcal V}$ then $\nu^{m_0}$ a smooth Mather measure.
  
  In Subsection \ref{subsec.conclusion.delta} we collect all the previous results and we finally prove that the limit $\chi_0$ of $\bar{\mathcal V}_\delta$ is uniquely defined by
  
  $$
  \chi_0(m)=\sup_{\chi\in{\mathcal S}^-}\chi(m).
  $$

\section{Assumptions and preliminary results}\label{sec.defin.delta}
\subsection{Notation and assumptions}

As it was mentioned in the introduction, this paper is meant to expand the weak KAM theory in the context of MFG introduced in \cite{cardaliaguet2019weak} and \cite{masoero2019}. Therefore, we will suppose that the very same assumptions are in place.
\begin{remark} Note that some of these hypothesis will not be explicitly used in this paper; especially the ones on the growth of the Hamiltonian and its derivatives. Nonetheless, as we give for granted several results of \cite{cardaliaguet2019weak,masoero2019} we still need to impose them.  
\end{remark}

  We will use as state space the $d-$dimensional flat torus $\T^{d}=\R^{d}/\Z^{d}$. This domain is chosen to avoid boundary conditions and to set the problem on a compact domain. 
We denote respectively by $\mathcal M(\T^d)$ and $\mathcal P(\T^{d})$ the set of Borel measures  and probability measures on $\T^d$. The set $\mathcal{P}(\T^d)$ is a compact, complete and separable metric space when endowed with the $1$-Wasserstein distance $\mathbf d(\cdot,\cdot)$. 
If $m\in C^0([t,T],\mathcal P(\T^d))$ then we set $L^{2}_{m}([t,T]\times\T^{d})$ the set of $m$-measurable functions $f$ such that the integral of $|f|^{2}dm(s)$ over $[t,T]\times\T^{d}$ is finite.

We use the notion of derivative on the metric space $\mathcal P(\T^{d})$ introduced in \cite{cardaliaguet2015master}. We say that $\Phi:\mathcal P(\T^{d})\rightarrow \R$ is $C^{1}$ if there exists a continuous function $\frac{\delta \Phi}{\delta m}:\mathcal P(\T^{d})\times\T^{d}\rightarrow\R$ such that
$$
\Phi(m_{1})-\Phi(m_{2})=\int_{0}^{1}\int_{\T^{d}} \frac{\delta \Phi}{\delta m}((1-t)m_{1}+tm_{2},x)(m_{2}-m_{1})(dx)dt,\qquad\forall m_{1},m_{2}\in\mathcal P(\T^{d}).
$$
As this derivative is defined up to an additive constant, we use the standard normalization
\begin{equation}\label{deriv.conve}
\int_{\T^{d}} \frac{\delta \Phi}{\delta m}(m,x)m(dx)=0.
\end{equation}
If $\frac{\delta \Phi}{\delta m}$ is derivable in the second variable, we define the \textit{intrinsic derivative}
$$
D_m \Phi(m,x)=D_x \frac{\delta \Phi}{\delta m}(m,x).
$$ 
Moreover, we say that a function $\Phi$ belongs to $C^{1,1}(\T^d)$ if $\Phi\in C^1(\T^d)$, $D_m\Phi$ is well defined and $D_m\Phi$ is smooth in the second variable.

We recall that, if $\mu$, $\nu\in\mathcal P(\T^d)$, the $1$-Wasserstein distance is defined by
\begin{equation}\label{RKd2}
\mathbf d(\mu,\nu)= \sup \left\{ \left. \int_{\T^{d}} \phi(x) \,d (\mu - \nu) (x) \right| \mbox{continuous } \phi : \T^{d} \to \mathbb{R},\, \mathrm{Lip} (\phi) \leq 1 \right\}.
\end{equation}

\textbf{Assumptions:} Throughout the paper we will suppose the following conditions:
\begin{enumerate}
	\item $H:\T^d\times\R^d\rightarrow\R$ is of class $C^2$, $p\mapsto D_{pp}H(x,p)$ is Lipschitz continuous, uniformly with respect to $x$. Moreover, there exists $\bar C>0$ that verifies 
	\begin{equation}\label{1Vdelta}
	\bar C^{-1}I_d\leq D_{pp}H(x,p)\leq \bar CI_d, \quad\forall (x,p)\in\T^d\times\R^d
	\end{equation}
	and $\theta\in(0,1)$, $C>0$ such that the following conditions hold true
	\begin{equation}\label{2Vdelta}
	|D_{xx}H(x,p)|\leq C(1+|p|)^{1+\theta},\quad|D_{x,p}H(x,p)|\leq C(1+|p|)^\theta,\quad\forall (x,p)\in\T^d\times\R^d.
	\end{equation}
	
	\item $\mathcal F:\mathcal P(\T^{d})\rightarrow\R$ is of class $C^2$. Its derivative $F:\T^d\times\mathcal P(\T^d)\rightarrow\R$ is twice differentiable in $x$ and $D^2_{xx}F$ is bounded.
	
\end{enumerate}

\subsection{Definitions and preliminary results}
We define ${\mathcal V}_\delta:\mathcal P(\T^d)\rightarrow \R$ as
\begin{equation}\label{def.Vdelta}
{\mathcal V}_\delta(m_0)=\inf_{(m,\alpha)}\int_0^\infty e^{-\delta t}\int_{\T^d}H^*(x,\alpha)dm(t)+\mathcal F(m(t))dt,\quad m_0\in\mathcal P(\T^d)
\end{equation}
where $\delta>0$,  $m\in C^0([0,+\infty),\mathcal P(\T^d))$, $\alpha\in L^2_{m,\delta}([0,+\infty)\times\T^d,\R^d)$, that is $L_{m}^{2}$ with weight $e^{-\delta t}$, and $(m,\alpha)$ verifies on $[0,+\infty)\times\T^{d}$ the Fokker-Plank equation
\begin{equation}\label{fokker.plank}
\begin{cases}
-\partial_t m +\Delta m+{\rm div}(m\alpha)=0\\
m(0)=m_0.
\end{cases}
\end{equation}
Standard arguments in optimal control ensures that  $\mathcal V_\delta$ solves the following dynamic programming principle
\begin{align}\label{Vdelta.ddp.begin}
{\mathcal V}_\delta(m_0)=\inf_{(m,\alpha)}\int_0^t e^{-\delta s}\int_{\T^d}H^*(x,\alpha(s))dm(s)+\mathcal F(m(s))ds+e^{-\delta t}{\mathcal{V}}_\delta(m(t)).
\end{align}
where $(m,\alpha)$ are defined as before. 

As we have anticipated in the introduction we will characterized the limit of $\mathcal V_\delta$ through a special class of subsolutions of the critical equation 
\begin{equation}\label{critical.HJ}
-\int_{\T^d}\dive_y D_m\chi(m,y) m(dy)+\int_{\T^d} H(y, D_m\chi(m,y))m(dy)-\mathcal F(m)=\lambda,\qquad m\in\mathcal P(\T^d),
\end{equation}
where the value $\lambda$ is the one discussed in the introduction. 

We will most often work with functions that do not enjoy enough regularity to solve in classical sense the above equation. Therefore, we need to introduce a weaker definition subsolution and, accordingly, the notion of corrector function. 
\begin{defin}\label{def.HJ.subsol}
	We say that a continuous function $\chi$ on $\mathcal{P}(\T^d)$ is a subsolution of the critical equation \eqref{critical.HJ} if, for any $h>0$ and $m_0\in\mathcal P(\T^d)$,
	\begin{equation}\label{subsol.ine}
	\chi(m_0)\leq\inf_{(m,\alpha)}\int_0^h\int_{\T^d}H^*(x,\alpha(t))dm(t)+\mathcal F(m(t))dt +\chi(m(h))+\lambda h, 
	\end{equation}
	where $m\in C^0([0,h],\mathcal P(\T^d))$, $\alpha\in L^2_{m}([0,h]\times\T^d,\R^d)$ and $(m,\alpha)$ verifies \eqref{fokker.plank} in $[0,h]$ with initial condition $m(0)=m_0$.
	
	If, otherwise, $\chi$ verifies \eqref{subsol.ine}, for any $m_0\in\mathcal P(\T^d)$ and any $h>0$, as an equality, we say that $\chi$ is corrector function.
\end{defin}
We know from \cite[Proposition 1.6]{masoero2019} that $\mathcal V_\delta$ is uniformly Lipschitz continuous with respect to $\delta$ and that $\delta\mathcal V_\delta\rightarrow-\lambda$. We define
$$
\bar{\mathcal V}_\delta(m)=\mathcal V_\delta(m)+\frac{\lambda}{\delta}.
$$

It is clear that $\delta \bar{\mathcal V}_\delta\rightarrow 0$. Moreover we claim that if $c\in\R$ is such that $\mathcal V_\delta+{c}/{\delta}$ is uniformly bounded, then $c=\lambda$. We also claim that $\bar{\mathcal V}_\delta$ converges, at least up to subsequence, to a corrector $\chi_0$ (Definition \ref{def.HJ.subsol}).  The proof of these claims is postponed in Lemma \ref{bound.barV} for which we need some preliminary results.

We will use more than once the fact that, for any $m_0\in\mathcal P(\T^d)$ and any $\delta>0$, the minimization problem \eqref{def.Vdelta} admits minimizer $(m_\delta,\alpha_\delta)$ and, as proved in \cite[Proposition 1.1]{masoero2019}, that $\alpha_\delta(t,x)=D_p H(x,Du_\delta(t,x))$ where, $u_\delta\in C^{1,2}([0,+\infty)\times\T^d)$ and $(u_\delta,m_\delta)$ solves

\begin{equation}\label{delta.system}
\begin{cases} 
-\partial_t u-\Laplace u+\delta u+H(x,Du)=F(x,m) & \mbox{in } \T^d\times[0,+\infty)\\
-\partial_t m+\Laplace m+{\rm div}(mD_pH(x,Du))=0 & \mbox{in }\T^d\times[0,+\infty)\\
m(0)=m_0,\; u\in L^{\infty}([0,+\infty)\times\T^{d}).
\end{cases}
\end{equation}

Moreover, $(m_\delta,\alpha_\delta)$ enjoys the following estimates (for the proof see \cite[Lemma 1.3]{masoero2019}):
\begin{lemma}\label{2estVdelta} There exists $C_1>0$ independent of $m_0,\delta$ such that, if $(u_\delta,m_\delta)$ is a classical solution of \eqref{delta.system}, then
	
	\begin{itemize}
		\item $\Vert D u_\delta\Vert_{L^\infty([0,+\infty)\times\T^d)}\leq C_1$
		\item $\Vert D^2u_\delta\Vert_{L^\infty([0,+\infty)\times\T^d)}\leq C_1$.
		\item $\bm d(m_\delta(s),m_\delta(l))\leq C_1 |l-s|^{1/2}$ for any $l,s\in [0,+\infty)$
	\end{itemize}
	
	Consequently, we also have that $\Vert\partial_{t}u_\delta\Vert_{L^\infty([0,+\infty)\times\T^d)}\leq C_1$ for any $s\in[0,+\infty)$.
\end{lemma}

\begin{defin}\label{def.mather.measure} We call Mather measure any minimizer of the following minimization problem 
	\begin{equation}\label{dual.min.pb}
	 \inf_{(\mu, p_1)}\int_{{\mathcal P}(\T^d)} \int_{\T^d} H^*\Bigl(y, \frac{dp_1}{d m\otimes \mu}\Bigr) m(dy)+\mathcal F(m)\,\mu(dm), 
	\end{equation}
	where the minimum is taken over $\mu\in {\mathcal P}(\mPT)$ and $p_1$, Borel vector measure on $\mPT\times \T^d$ with value in $\R^d$, such that $p_1$ is absolutely continuous with respect to the measure $dm\otimes \mu:=m(dy)\mu(dm)$ and such that $(\mu, p_1)$ is closed, in the sense that, for any $\Phi\in C^{1,1}(\mathcal P(\T^d))$, 
	\begin{equation}\label{close.rel}
	-\int_{{\mathcal P}(\T^d)\times \T^d} D_m\Phi(m,y) \cdot p_1(dm,dy) +\int_{{\mathcal P}(\T^d)\times \T^d} \dive_y D_m\Phi(m,y) m(dy)\mu(dm) =0.
	\end{equation}
\end{defin} 

As in the standard Aubry-Mather theory, in \cite{cardaliaguet2019weak} it was proven that, if $(\mu,p_1)$ is a Mather measure, then 
\begin{align}\label{inf.equal.lambda}
	\int_{{\mathcal P}(\T^d)} \int_{\T^d} H^*\Bigl(y, \frac{dp_1}{d m\otimes \mu}\Bigr) m(dy)+\mathcal F(m)\,\mu(dm)=-\lambda.
\end{align}

\section{Convergence of $\bar{\mathcal V}_\delta$}\label{sec.conv.Vdelta}

\subsection{Measures induced by optimal trajectories}\label{subsec.meas.ind}

Let  $(m_\delta,\alpha_\delta)$ be an optimal trajectory for $\bar{\mathcal V}_\delta(m_0)$, then if $C_1>0$ is the constant that appears in Lemma \eqref{2estVdelta}, we have $\Vert \alpha_\delta\Vert_{W^{1,\infty}(\T^d\times[0,+\infty))}\leq C_1$. 

We define $\nu_\delta^{m_0}\in\mathcal P(\T^d\times E)$ as follows
\begin{align}\label{defin.numdelta}
\int_{{\mathcal P}(\T^d)\times E}f(m,\alpha)\nu_\delta^{m_0}(dm,d\alpha)=\delta\int_0^\infty e^{-\delta s}f(m^\delta(s),\alpha^\delta(s))ds,
\end{align}

where 
$$
E:= \{ \alpha \in W^{1,\infty}(\T^d, \R^d), \; \|\alpha\|_\infty+\|D\alpha\|_\infty\leq C\}.
$$
Note that, we know from Lemma \ref{2estVdelta} that we can chose $C$ independent of $\delta$.
As $\mathcal P(\T^d)\times E$ is compact when endowed with the uniform convergence, we can suppose that $\nu_\delta^{m_0}$ weakly converges, up to subsequence, to a probability measure $\bar\nu$ on $E$. 
Let $\mu$ be the first marginal of $\bar\nu$ and let us define the vector measure $p_1$ on $\mPT\times \T^d$ as 
$$
\int_{\mPT\times \T^d} \phi(m,y)\cdot p_1(dm, dy) = \int_{\mPT\times E} \int_{\T^d} \phi(m,y)\cdot \alpha(y) m(dy) \bar\nu(dm,d\alpha)
$$
for any test function $\phi\in C^0(\mPT\times \T^d, \R^d)$. Note that $p_1$ is absolutely continuous with respect to $\mu$, since, if we disintegrate $\bar\nu$ with respect to $\mu$: $\bar\nu=\bar\nu_m(d\alpha)\mu(dm)$, then 
\begin{align}\label{p1.defin}
p_1(dm, dy)= \int_{E} \alpha(y) m(dy) \bar\nu_m(d\alpha) \mu(dm)
\end{align}
and so
\begin{align}\label{p1dev.defin}
\frac{dp_1}{dm\otimes \mu}(m,y)= \int_{E} \alpha(y)\bar\nu_m(d\alpha). 
\end{align}

In the next proposition we prove that the couple $(\mu,p_1)$ defined above is a Mather measure.

\begin{prop}\label{nu.mather}
	Let $(m_\delta,\alpha_\delta)$ be an optimal trajectory for $\bar{\mathcal V}_\delta(m_0)$ and $\nu_\delta^{m_0}$ be the probability measure defined by \eqref{defin.numdelta}. If $\bar\nu\in\mathcal{P}(\T^d)$ is a weak limit of $\nu_\delta^{m_0}$ (possibly up to subsequence), $\mu$ is the first marginal of $\bar \nu$ and $p_1$ is defined as \eqref{p1.defin}, then $(\mu,p_1)$ is a Mather measure. 
\end{prop}
\begin{proof}
We first need to check that $(\mu,p_1)$ is closed in the sense of \eqref{close.rel}.
 Let $\Phi\in C^{1,1}({\mathcal P}(\T^d))$, then 
\begin{align*}
\frac{d}{dt} \Phi(m_\delta(t))= \inte  \dive(D_m\Phi(m_\delta(t),y))m_\delta(t,dy) + \inte D_m\Phi(m_\delta(t),y)\cdot \alpha_\delta(t,y) m_\delta(dy).
\end{align*}
So,
\begin{align}\label{}
\int_0^t\inte  \dive(D_m\Phi(m_\delta(t),y))m_\delta(t,dy) + \inte D_m\Phi(m_\delta(t),y)\cdot \alpha_\delta(t,y) m_\delta(dy)dt=\Phi(m_\delta (t))-\Phi(m_0).
\end{align}
Using the above relation and integrating by parts, we get
\begin{align*}
& \left\vert\int_{\mPT\times E}  \inte \dive(D_m\Phi(m,y)) +  D_m\Phi(m,y)\cdot \alpha(y) m(dy) \ \nu_\delta^{m_0}(dm,d\alpha) \right\vert \\
& = \delta \left\vert\int_0^{+\infty} e^{-\delta t}\inte  \dive(D_m\Phi(m_\delta(t),y))m_\delta(t,dy) + \inte D_m\Phi(m_\delta(t),y)\cdot \alpha_\delta(t,y) m_\delta(dy) dt\right\vert \\ 
&\leq \delta \left\vert\lim_{T\rightarrow+\infty} e^{-\delta T}(\Phi(m_\delta(T))-\Phi(m_\delta(m_0)) \right\vert+\delta^2\left\vert \int_0^{+\infty}e^{-\delta t}(\Phi(m_\delta(t))-\Phi(m_\delta(m_0))dt\right\vert\\
& \leq \delta^2K\int_0^{+\infty}e^{-\delta t}dt=\delta K.  
\end{align*}
Letting $\delta\to0$ we find
\begin{align*}
& \int_{\mPT\times E}  \inte \dive(D_m\Phi(m,y)) +  D_m\Phi(m,y)\cdot \alpha(y) m(dy) \ \bar\nu(dm,d\alpha)=0.
\end{align*}
According to the definition of $p_1$ we can read the last equality as
\begin{align*}
& \int_{\mPT}  \inte \dive(D_m\Phi(m,y)) m(dy) \mu(dm) + \int_{\mPT\times \T^d}  D_m\Phi(m,y)\cdot p_1(dm,dy)=0,
\end{align*}
which proves that $(\mu,p_1)$ is closed. The last step is to prove that $(\mu,p_1)$ is a minimizer of \eqref{dual.min.pb}. Indeed, by convexity of $H^*$, we have
\begin{align}\label{}
&\int_{\mPT} \left[\int_{\T^d} H^*\left(y, \frac{dp_1}{dm\otimes \mu} (m,y)\right)m(dy) + {\mathcal F}(m)\right]\ \mu(dm)\\
&=\int_{\mPT} \left[\int_{\T^d} H^*\left(y, \int_E\alpha(y)\bar\nu_m(d\alpha)\right)m(dy) + {\mathcal F}(m)\right]\ \mu(dm)\\
&\leq\int_{\mPT}\int_E \left[\int_{\T^d} H^*(y, \alpha(y))m(dy) + {\mathcal F}(m)\right]\ \bar\nu_m(d\alpha)\mu(dm)\\
&=\int_{\mPT\times E} \left[\int_{\T^d} H^*(y, \alpha(y))m(dy) + {\mathcal F}(m)\right]\ \bar\nu(dm,d\alpha) \\
&=\lim_{\delta\rightarrow0}\int_{\T^d} H^*(y, \alpha(y))m(dy) + {\mathcal F}(m)\ \nu_\delta^{m_0}(dm,d\alpha)=\lim_{\delta\rightarrow0}\delta \mathcal{V}_\delta (m_0)=-\lambda.
\end{align}

\end{proof}
\subsection{Lower bound}\label{subsec.lower.bound}
In this section we will first prove the analogue of \cite[Lemma 3.5]{davini2016convergence} and then the boundedness of $\bar{\mathcal{V}}_\delta$. Note that, the lack of a proper mollification for functions defined on $\mathcal P(\T^d)$, forced us to find a proof that differs from the one in \cite{davini2016convergence}. While \cite{davini2016convergence} used approximation of solutions of the critical equation, we work here at the level of optimal trajectories.

\begin{prop}\label{punto.3}
	Let $m_0\in\mathcal P(\T^d)$, $(m_\delta,\alpha_\delta)$ be an optimal trajectory for $\bar{\mathcal V}_\delta(m_0)$ and $\nu_\delta^{m_0}$ be the probability measure defined by \eqref{defin.numdelta}. Then, for any $\chi$ subsolution of \eqref{critical.HJ},
	$$
	\bar{\mathcal{V}}_\delta(m_0)\geq\chi(m_0)-\int_{\mathcal P(\T^d)}\chi(m)\nu_\delta^{m_0}(dm).
	$$
\end{prop}
\begin{proof}
	The dynamic programming principle \eqref{Vdelta.ddp.begin} says that, if $(m_\delta,\alpha_\delta)$ is an optimal trajectory for $\bar{\mathcal V}_\delta (m_0)$, then, for any $t>0$, $(m_\delta,\alpha_\delta)$ is a minimizer of 
	\begin{align}\label{Vdelta.ddp}
	\bar{\mathcal V}_\delta(m_0)=\inf_{(m,\alpha)}\int_0^t e^{-\delta s}\int_{\T^d}H^*(x,\alpha(s))dm(s)+\mathcal F(m(s))ds+e^{-\delta t}\bar{\mathcal{V}}_\delta(m(t))+\frac{\lambda}{\delta}.
	\end{align}
	
	In order to keep the notation as simple as possible we define 
	$$
	\mathcal{L}(t,s)=\int_{{\mathcal P}(\T^d)}H^*(\alpha_\delta(t+s))dm_\delta(t+s)+\mathcal{F}(m_\delta(t+s))+\lambda
	$$
	and
	$$
	\mathbb{L}_\delta(t,h)=\int_0^he^{-\delta s}\mathcal L(t,s)ds,
	$$
	so that, according to the dynamic programming principle \eqref{Vdelta.ddp},
	\begin{equation}\label{serv.Vdelta}
	\mathbb{L}_\delta(t,h)=\bar{\mathcal V}_\delta(m_\delta(t))-e^{-\delta h}\bar{\mathcal V}_\delta(m_\delta(t+h)).
	\end{equation}
	We start with the following computation
	$$
	\int_0^h\mathcal{L}(t,s)ds=\int_0^he^{\delta s}e^{-\delta s}\mathcal{L}(t,s)ds=e^{\delta h}\mathbb{L}_\delta(t,h)-\delta\int_0^he^{\delta s}\mathbb{L}_\delta(t,s)ds.
	$$
	Plugging \eqref{serv.Vdelta} into the last equality, we get
	\begin{equation}\label{serv2.Vdelta}
	\int_0^h\mathcal{L}(t,s)ds=e^{\delta h}\bar{\mathcal V}_\delta(m_\delta(t))-\bar{\mathcal V}_\delta(m_\delta(t+h))-\delta\int_0^he^{\delta s}\bar{\mathcal V}_\delta(m_\delta(t))-\bar{\mathcal V}_\delta(m_\delta(t+s))ds.
	\end{equation}
	Now we can focus on $\int_{\mathcal P(\T^d)}\chi(m)\nu_\delta^{m_0}(dm)$. Using the above relations and the definitions of $\nu_\delta^{m_0}$ and subsolution of \eqref{critical.HJ}, we get
	$$
	\int_{\mathcal P(\T^d)}\chi(m)\nu_\delta^{m_0}(dm)=\delta\int_0^{+\infty}e^{-\delta t}\chi(m_\delta(t))dt\leq\delta\int_0^{+\infty}e^{-\delta t}\int_0^h\mathcal L(t,s)ds+\chi(m_\delta(t+h))dt.
	$$
	Arranging the terms and dividing by $\delta$, we find
	\begin{equation}\label{serv.left}
	\int_0^{+\infty}e^{-\delta t}\chi(m_\delta(t))dt-\int_0^{+\infty}e^{-\delta t}\chi(m_\delta(t+h))dt\leq\int_0^{+\infty}e^{-\delta t}\int_0^h\mathcal L(t,s)dsdt.
	\end{equation}
	We first consider the left hand side. If we run a change of variable we get
	\begin{align}\label{serv.Vdelta.left}
	&\int_0^{+\infty}e^{-\delta t}\chi(m_\delta(t))dt-\int_0^{+\infty}e^{-\delta t}\chi(m_\delta(t+h))dt\\
	&=\int_0^{+\infty}e^{-\delta t}\chi(m_\delta(t))dt-e^{\delta h}\int_{h}^{+\infty}e^{-\delta t}\chi(m_\delta(t))dt\\
	&=(1-e^{\delta h})\int_{0}^{+\infty}e^{-\delta (t)}\chi(m_\delta(t))dt+e^{\delta h}\int_{0}^{h}e^{-\delta t}\chi(m_\delta(t))dt.
	\end{align}
	
	We now work on the right hand side of \eqref{serv.left}. Using again \eqref{serv2.Vdelta} we get
	
	\begin{align}\label{}
	&\int_0^{+\infty}e^{-\delta t}\int_0^h\mathcal L(t,s)dsdt\\
	&=\int_0^{+\infty}e^{-\delta t}\left(e^{\delta h}\bar{\mathcal V}_\delta(m_\delta(t))-\bar{\mathcal V}_\delta(m_\delta(t+h))\right)dt-\delta\int_0^{+\infty}e^{-\delta t}\int_0^he^{\delta s}\bar{\mathcal V}_\delta(m_\delta(t))-\bar{\mathcal V}_\delta(m_\delta(t+s))dsdt.
	\end{align}
	We now look separately the two addends of the above line. First we have that
	\begin{align}\label{ser.Vdelta.fin0}
	&\int_0^{+\infty}e^{-\delta t}\left(e^{\delta h}\bar{\mathcal V}_\delta(m_\delta(t))-V_\delta(m_\delta(t+h))\right)dt=\int_0^{+\infty}e^{-\delta (t-h)}\bar{\mathcal V}_\delta(m_\delta(t))-\int_{h}^{+\infty}e^{-\delta (t-h)}\bar{\mathcal V}_\delta(m_\delta(t))dt\\
	&=\int_{0}^he^{-\delta (t-h)}\bar{\mathcal V}_\delta(m_\delta(t))dt=e^{\delta h}\int_0^he^{-\delta t}\bar{\mathcal V}_\delta (m_\delta(t))dt,
	\end{align}
	then,
	\begin{align}\label{serv.Vdelta.fin}
	&\left\vert\int_0^{+\infty}e^{-\delta t}\int_0^he^{\delta s}\bar{\mathcal V}_\delta(m_\delta(t))-\bar{\mathcal V}_\delta(m_\delta(t+s))dsdt\right\vert\leq\\
	&\int_0^{+\infty}e^{-\delta t}\int_0^he^{\delta s}\left\vert \bar{\mathcal V}_\delta(m_\delta(t))-e^{-\delta s}\bar{\mathcal V}_\delta(m_\delta(t+s))\right\vert dsdt\leq\\
	& e^{\delta h}\int_0^{+\infty}e^{-\delta t}\int_0^h\left\vert \bar{\mathcal V}_\delta(m_\delta(t))-e^{-\delta s}\bar{\mathcal V}_\delta(m_\delta(t+s))\right\vert dsdt.
	\end{align}
	Note that, as $\bar{\mathcal V}_\delta$ is Lipschitz continuous. Moreover, we know from Lemma \ref{2estVdelta} that there exists a $C>0$ independent of $m_0$ and $t$ such that
	$$
	\bm d(m_\delta (t+h),m_\delta(t))\leq C h^{\frac{1}{2}}.
	$$
	Then, for small $h>0$, if $s\leq h$
	\begin{align}
	&\left\vert \bar{\mathcal V}_\delta(m_\delta(t))-e^{-\delta s}\bar{\mathcal V}_\delta(m_\delta(t+s))\right\vert=\left\vert (1-e^{-\delta s}) \bar{\mathcal V}_\delta(m_\delta(t))-e^{-\delta s}(\bar{\mathcal V}_\delta(m_\delta(t+s))-\bar{\mathcal V}_\delta(m(t)))\right\vert\leq\\
	&\left\vert \bar{\mathcal V}_\delta(m_\delta(t))(1-e^{-\delta s})\right\vert+\left\vert (\bar{\mathcal V}_\delta(m_\delta(t+s))-\bar{\mathcal V}_\delta(m(t)))\right\vert\leq K (h+\bm{d}(m_\delta(t+s),m_\delta(t)))\leq C_1h^{\frac{1}{2}}.
	\end{align}
	Coming back to \eqref{serv.Vdelta.fin}, we get
	\begin{align}\label{serv.Vdelta.fin2}
	&\left\vert\int_0^{+\infty}e^{-\delta t}\int_0^he^{\delta s}\bar{\mathcal V}_\delta(m_\delta(t))-\bar{\mathcal V}_\delta(m_\delta(t+s))dsdt\right\vert\\
	&\leq e^{\delta h}\int_0^{+\infty}e^{-\delta t}C_1 h\, h^{\frac{1}{2}}dt=\frac{e^{\delta h}}{\delta}C_1 h\, h^{\frac{1}{2}}.
	\end{align}
	If we plug \eqref{serv.Vdelta.left}, \eqref{ser.Vdelta.fin0} and \eqref{serv.Vdelta.fin2} into \eqref{serv.left}, we get
	\begin{align}\label{}
	&(1-e^{\delta h})\int_{0}^{+\infty}e^{-\delta (t)}\chi(m_\delta(t))dt+e^{\delta h}\int_{0}^{h}e^{-\delta t}\chi(m_\delta(t))dt\\
	&\leq e^{\delta h}\int_0^he^{-\delta t}\bar{\mathcal V}_\delta (m_\delta(t))dt+\frac{e^{\delta h}}{\delta}C_1 h\, h^{\frac{1}{2}}.
	\end{align}
	Now, we just need to divide by $h$ and let $h$ go to $0$ to find
	\begin{align}\label{}
	-\delta\int_0^{+\infty}e^{-\delta t}\chi(m_\delta (t))dt+\chi(m_0)\leq\bar{\mathcal V}_\delta(m_0).
	\end{align}
\end{proof}
We can now prove that $\bar{\mathcal{V}}$ does actually converge, up to subsequence, to a corrector $\chi_0$.
\begin{lemma}\label{bound.barV}
	The family of function $\{\bar{\mathcal V}_\delta\}_\delta$ is uniformly bounded, uniformly Lipschitz continuous and therefore admits, up to subsequence, a uniform limit $\chi_0$. Moreover, $\chi_0$ is a corrector function.
\end{lemma}

\begin{proof}
	The only thing that we have to prove is that $\bar{\mathcal V}_\delta$ is uniformly bounded. Indeed, the rest of the claim was proven in \cite[Proposition 1.6]{masoero2019} and \cite[Lemma 1.7]{masoero2019}. 
	
	We start showing that $\bar{\mathcal{V}}_\delta$ is bounded below. Let $\chi$ be a corrector function and $\bar m\in\mathcal P(\T^d)$ be such that $\max_{m\in\mathcal P(\T^d)}\chi(m)=\chi(\bar m)$. From Proposition \ref{punto.3}, if we fix $m_0\in\mathcal P(\T^d)$ we get
	$$
	\bar{\mathcal{V}}_\delta(m_0)\geq\chi(m_0)-\int_{\mathcal P(\T^d)}\chi(m)\nu_\delta^{m_0}(dm).
	$$
	Then,
	$$
	\bar{\mathcal{V}}_\delta(m_0)\geq\chi(m_0)-\int_{\mathcal P(\T^d)}\max_{m\in\mathcal P(\T^d)}\chi(m)\nu_\delta^{m_0}(dm)
	$$
	$$
	=\chi(m_0)-\chi(\bar m)\geq-K\bm d(m_0,\bar m)\geq -K{\rm diam}(\mathcal P(\T^d))
	$$
	and the Lipschitz constant $K$ does not depend on $\chi$.
	
	 We now focus on the upper bound. We fix $(m,\alpha)$ an optimal trajectory for $\chi(m_0)$ which means that, for any $t>0$, 
	 \begin{align}\label{corrector.ddp}
	 \chi(m_0)-\chi(m(t))=\int_0^t\int_{\T^d}H^*(x,\alpha(s))dm(s)+\mathcal F(m(s))ds +\lambda t.
	 \end{align}

	  To have a lighter notation we introduce
	$$
	\mathcal{L}(s)=\int_{{\mathcal P}(\T^d)}H^*(x,\alpha(s))dm(s)+\mathcal{F}(m(s))
	$$
	and
	$$
	\mathbb{L}(t)=\int_0^t\mathcal L(s)ds.
	$$
	Note that according to \eqref{corrector.ddp} we have
	\begin{align}\label{serv.L}
	\mathbb{L}(t)=\chi(m_0)-\chi(m(t))-t\lambda.
	\end{align}
	Moreover, we recall that $\bar{\mathcal{V_\delta}}$ verifies the dynamical programming principle \eqref{Vdelta.ddp}. Therefore,
	
	\begin{align}\label{}
	\bar{\mathcal{V}}_\delta(m_0)\leq\int_0^Te^{-\delta s}\mathcal L(s)ds+e^{-\delta T}\bar{\mathcal{V}}_\delta(m(T))+\frac{\lambda}{\delta}.
	\end{align}
	Integrating by parts we find
	\begin{align}\label{serv.supine.Vdelta}
	\bar{\mathcal{V}}_\delta(m_0)\leq e^{-\delta T}\mathbb L(T)+\delta\int_0^Te^{-\delta s}\mathbb L(s)ds+e^{-\delta T}\bar{\mathcal{V}}_\delta(m(T))+\frac{\lambda}{\delta}.
	\end{align}
	We now consider each addends on the right hand side to prove that it either converges to zero when $T\rightarrow+\infty$ or it is uniformly bounded. From \eqref{serv.L}, we deduce that
	\begin{align}\label{}
	&e^{-\delta T}\mathbb L(T)=e^{-\delta T}(\chi(m_0)-\chi(m(T))-T\lambda)\xrightarrow[]{T\rightarrow +\infty}0.
	\end{align} 
	Then, there exists a $C>0$ such that
	\begin{align}\label{}
	&\delta\int_0^Te^{-\delta s}\mathbb L(s)ds=\delta\int_0^Te^{-\delta s}(-s\lambda+\chi(m_0)-\chi(m(s)))ds=\\
	&-(1-e^{-\delta T})\frac{\lambda}{\delta}+\delta\int_0^Te^{-\delta s}(\chi(m_0)-\chi(m(s)))ds\leq\\
	&-(1-e^{-\delta T})\frac{\lambda}{\delta}+(1-e^{-\delta T})C\xrightarrow[]{T\rightarrow +\infty}-\frac{\lambda}{\delta}+ C,
	\end{align}
	where in the last inequality we used the boundedness of $\chi$.
	Note that, for fixed $\delta$, also the term $e^{-\delta T}\bar{\mathcal{V}}_\delta(m(T))$ converges to zero when $T\rightarrow +\infty$. Then, plugging these computations into \eqref{serv.supine.Vdelta} and letting $T\rightarrow+\infty$, we get
	\begin{align}\label{}
	\bar{\mathcal{V}}_\delta(m_0)\leq -\frac{\lambda}{\delta}+ C +\frac{\lambda}{\delta}=C.
	\end{align}
\end{proof}

\subsection{Smooth Mather measures}\label{subsec.smooth.mather}
In this section we will work with measures with a smooth density. In this case, we will often identify a measure $m\in\mathcal P(\T^d)$ with its density. 
\begin{defin}\label{defin.smooth.measure}
	We say that $(\mu,p_1) \in \mathcal P(\mathcal P(\T^d))\times\mathcal M(\mathcal P(\T^d)\times E,\R^d)$ is smooth if there exists $C>0$ such that, for any $m\in{\rm supp}\,\mu$,
	\begin{enumerate}
		\item   $\Vert m\Vert_{C^2(\T^d)}\leq C$.
		\item   $\left\Vert dp_1/(dm\otimes\mu)(\cdot,m)\right\Vert_{C^1(\T^d)}\leq C$.
	\end{enumerate}
\end{defin}
\begin{lemma}\label{lemma.smooth.measure}
	Let $(\mu,p_1)$ be a smooth (in the sense Definition \ref{defin.smooth.measure}) closed measure. Then, if $\Phi\in C^0(\mathcal P(\T^d))$,
	\begin{align}\label{}
	\int_{{\mathcal P}(\T^d)} \left(\Delta m(x)+\dive\left(\frac{dp_1}{dm\otimes\mu}(x,m)m(x)\right)\right)\Phi(m)\mu(dm)=0
	\end{align}
\end{lemma}
\begin{proof}
	First of all, note that the left hand side of the above relation is integrable because of the bounds on the space derivatives of $m$ and $p_1/dm\otimes\mu$.
	
	 According to \cite[Theorem 2.2]{mou2019weak}, we can pick a sequence $\Phi_n\in C^{1,1}$ which converges uniformly to $\Phi$. Then, using that $(\mu,p_1)$ is closed and smooth we get
	\begin{align}\label{}
	0&=-\int_{{\mathcal P}(\T^d)\times \T^d} D_m\Phi_n(m,y) \cdot p_1(dm,dy) +\int_{{\mathcal P}(\T^d)\times \T^d} \dive_y D_m\Phi_n(m,y) m(dy)\mu(dm)\\
	&=-\int_{{\mathcal P}(\T^d)\times \T^d} D_m\Phi_n(m,y) \cdot \frac{dp_1}{dm\otimes\mu}(x,m)dm(x)\mu(dm) +\int_{{\mathcal P}(\T^d)\times \T^d} \dive_y D_m\Phi_n(m,y) m(dy)\mu(dm)\\
	&=\int_{{\mathcal P}(\T^d)\times \T^d}\left(\Delta m(x)+\dive\left(\frac{dp_1}{dm\otimes\mu}(x,m)m(x)\right)\right)\Phi_n(m)\mu(dm),
	\end{align}
	that is
	\begin{align}\label{}
	0=\int_{{\mathcal P}(\T^d)\times \T^d}\left(\Delta m(x)+\dive\left(\frac{dp_1}{dm\otimes\mu}(x,m)m(x)\right)\right)\Phi_n(m)\mu(dm)
	\end{align}
	As $\Phi_n$ converges uniformly to $\Phi$ we just need to pass to the limit $n\rightarrow+\infty$.
\end{proof}

\begin{prop}\label{Vdelta.in.Mdelta}
	For any $\delta>0$ and any smooth Mather measure $(\mu,p_1)$, we have
	\begin{align}\label{}
	\int_{{\mathcal P}(\T^d)}\bar{\mathcal V}_\delta(m)\mu(dm)\leq0
	\end{align}
\end{prop}
\begin{proof}
	For any $\bar m\in\mathcal P(\T^d)$ there exists a smooth function $\Phi^{\bar m}\in C^{1,1}$ such that $\Phi^{\bar m}\in C^{1,1}(\mathcal P(\T^d))$, $\Phi^{\bar m}(m)\geq\bar{\mathcal V}_\delta(m)$ with an equality only for $m=\bar m$ and such that
	\begin{align}\label{}
	\delta\bar{\mathcal V}_\delta(\bar m)\leq\int_{\T^d}\dive_y D_m\Phi^{\bar m}(\bar m,y) \bar m(dy)-\int_{\T^d} H(y,D_m\Phi^{\bar m}(\bar m,y))\bar m(dy)+\mathcal F(\bar m)+\lambda.
	\end{align}
	For the construction of such a function one can look at \cite[Lemma 6.3]{cardaliaguet2019weak}.
	
	By convexity of $H$ with respect to the second variable, we get
	\begin{align}\label{}
	\delta\bar{\mathcal V}_\delta(\bar m)&\leq\int_{\T^d}\dive_y D_m\Phi^{\bar m}(\bar m,y)\bar m(dy)-\int_{\T^d}D_m\Phi^{\bar m}(\bar m,y)\cdot\frac{dp_1}{dm\otimes\mu}(\bar m,y) \bar m(dy)\\
	&\qquad+\int_{\T^d} H^*\left(y,\frac{dp_1}{dm\otimes\mu}(\bar m,y)\right)\bar m(dy)+\mathcal F(\bar m)+\lambda.
	\end{align}
	If $\bar m\in{supp}\,\mu$ then we can integrate by parts and we get
	\begin{align}\label{}
	\delta\bar{\mathcal V}_\delta(\bar m)\leq&\int_{\T^d}\left(\Delta \bar m(x)+\dive\left(\frac{dp_1}{dm\otimes\mu}(x,\bar m)\bar m(x)\right)\right)\Phi^{\bar m}(\bar m)\\
	&\qquad+\int_{\T^d} H^*\left(y,\frac{dp_1}{dm\otimes\mu}(\bar m,y)\right)\bar m(dy)+\mathcal F(\bar m)+\lambda\\
	&=\int_{\T^d}\left(\Delta \bar m(x)+\dive\left(\frac{dp_1}{dm\otimes\mu}(x,\bar m)\bar m(x)\right)\right)\bar{\mathcal V}_\delta(\bar m)\\
	&\qquad+\int_{\T^d} H^*\left(y,\frac{dp_1}{dm\otimes\mu}(\bar m,y)\right)\bar m(dy)+\mathcal F(\bar m)+\lambda
	\end{align}
	
	If we integrate against $\mu$ we get from Lemma \ref{lemma.smooth.measure} that 
	\begin{align}\label{}
	\int_{{\mathcal P}(\T^d)}\int_{\T^d}\left(\Delta \bar m(x)+\dive\left(\frac{dp_1}{dm\otimes\mu}(x,\bar m)\bar m(x)\right)\right)\bar{\mathcal V}_\delta(\bar m)\mu(d\bar m)=0.
	\end{align}
	Moreover, as $(\mu,p_1)$ is a Mather measure we also have
	\begin{align}\label{}
	\int_{{\mathcal P}(\T^d)}\int_{\T^d} H^*\left(y,\frac{dp_1}{dm\otimes\mu}(\bar m,y)\right)\bar m(dy)+\mathcal F(\bar m)\mu(d\bar m)=-\lambda.
	\end{align}
	Therefore,
	\begin{align}\label{}
	\delta\int_{{\mathcal P}(\T^d)}\bar{\mathcal V}_\delta(\bar m)\mu(d\bar m)\leq 0.
	\end{align}
\end{proof}
\begin{defin}
We say that $(\mu,p_1)\in \mathcal P(\mathcal P(\T^d))\times\mathcal M(\mathcal P(\T^d)\times E,\R^d)$ belongs to $\mathcal M_{\mathcal V}$ if $(\mu,p_1)$ is smooth in the sense of Definition \ref{defin.smooth.measure} and if it is the limit of $\nu_\delta^{m_0}$ (in the sense of Proposition \ref{nu.mather}) for a certain $m_0$ and along a subsequence $\delta\rightarrow 0$.
\end{defin}
Note that, according to Proposition \ref{nu.mather}, $\mathcal M_{\mathcal V}$ is a subset of the set of Mather measures. Moreover, we will say with an abuse of terminology that a measure $\nu\in\mathcal P(\mathcal P(\T^d)\times E)$ belongs to $\mathcal M_{\mathcal V}$ is the couple $(\mu,p_1)$ defined as in Proposition \ref{nu.mather} does.
\begin{lemma}\label{lemma.nu.Mdelta.lim}
For any $m_0\in\mathcal P(\T^d)$ with a $C^\infty(\T^d)$ density, there exists $\delta_n\rightarrow0$ such that $\nu_{\delta_n}^{m_0}\rightarrow\nu^{m_0}$ and $\nu^{m_0}\in\mathcal M_{\mathcal V}$.
\end{lemma}
\begin{proof}
	We know from Lemma \ref{2estVdelta} that, if $(m_\delta,\alpha_\delta)$ is a minimizer of $\bar{\mathcal{V}}_\delta$ then $\alpha_\delta$ is uniformly bounded, with respect to $\delta$, in $C^{1,1}([0,+\infty]\times\T^d)$. Therefore, $m_\delta$ solves the Fokker-Plank equation \eqref{fokker.plank} with a smooth drift. If, moreover, the initial condition $m_0$ is smooth, by standard parabolic estimates, we have that $m_\delta$ is uniformly bounded in $C^{1,2}([0,+\infty]\times\T^d)$. Then, as ${\rm supp}\,\nu_\delta^{m_0}=\{(m_\delta(t),\alpha_\delta (t))\}_{t\geq 0}$, there exists a constant $C>0$ such that for any $(m,\alpha)\in{\rm supp}\,\nu_\delta^{m_0}$, $\Vert(m,\alpha)\Vert_{C^2(\T^d)\times C^1(\T^d,\R^d)}\leq C$.
	
	Let now $\delta_n$ be a sequence such that $\nu_{\delta_n}^{m_0}\rightarrow\nu^{m_0}$. Given that ${\rm supp}\,\nu^{m_0}\subset\limsup_{\delta\rightarrow 0}{\rm supp}\,\nu_{\delta_n}^{m_0}$, the bounds that we have discussed ensure that the points of the support of $\nu^{m_0}$ are smooth in the sense of Definition \ref{defin.smooth.measure}.
\end{proof}

\subsection{Conclusion}\label{subsec.conclusion.delta}

We define ${\mathcal S}^-$ the set of subsolution $\Psi$ of \eqref{critical.HJ} such that for any $(\mu,p_1)\in\mathcal M_{\mathcal V}$ we have
\begin{align}\label{barS.defin}
\int_{{\mathcal P}(\T^d)}\Psi(m)\mu(dm)\leq0.
\end{align}

We set
\begin{align}\label{}
\bar\chi(m)=\sup_{\chi\in{\mathcal S}^-}\chi(m).
\end{align}

To give sense to the terms that appear in the above relation we need first to prove the following lemma.
\begin{lemma}
	The family of function $\mathcal S^-$ is not empty and uniformly bounded from above.
\end{lemma}
\begin{proof}
	Let $\chi$ be a corrector function and $C>0$ be such that $\chi-C<0$. If we set $\Psi=\chi-C$, then $\Psi$ is a corrector function and for any measure $(\mu,p_1)\in\mathcal M_{\mathcal V}$
	$$
	\int_{{\mathcal P}(\T^d)}\Psi(m)\mu(dm)\leq0.
	$$
	Therefore, $\mathcal S^-$ is not empty.	Following the structure of \cite[Theorem 1.5]{masoero2019}, one can easily prove that the set of subsolution of \eqref{critical.HJ} is uniformly Lipschitz continuous. Then, if we fix a $(\mu,p_1)\in\mathcal M_{\mathcal V}$ and a subsolution $\chi\in\mathcal S^-$, we have
	\begin{align}\label{}
	\min_{\nu\in\mathcal P(\T^d)}\chi(\nu)=\int_{{\mathcal P}(\T^d)}\min_{\nu\in\mathcal P(\T^d)}\chi(\nu)\mu(dm)\leq\int_{{\mathcal P}(\T^d)}\chi(m)\mu(dm)\leq 0.
	\end{align}
	If we use the Lipschitz continuity of $\chi$ we get that, for any $m\in\mathcal{P}(\T^d)$,
	\begin{align}\label{}
	\chi(m)\leq \min_{\nu\in\mathcal P(\T^d)}\chi(\nu)+ K{\rm diam}(\mathcal P(\T^d))\leq K{\rm diam}(\mathcal P(\T^d)),
	\end{align}
	which proves the claim.
\end{proof}
\begin{teo} The function $\bar{\mathcal V}_\delta$ uniformly converges to a corrector $\chi_0$, which is defined by
\begin{align}\label{teo.chi0.new}
\chi_0(m)=\sup_{\chi\in{\mathcal S}^-}\chi(m).
\end{align}
\end{teo}
\begin{proof}
	We fix a subsequence $\delta_n\rightarrow 0$ such that $\bar{\mathcal V}_{\delta_n}$ uniformly converges to a corrector $\chi_0$. The existence of such a subsequence was proven in Lemma \ref{bound.barV}. We know from Proposition \ref{Vdelta.in.Mdelta} that, for any $\delta_n>0$ and any smooth Mather measure $(\mu,p_1)$,
	\begin{align}\label{}
	\int_{{\mathcal P}(\T^d)}\bar{\mathcal V}_{\delta_n}(m)\mu(dm)\leq0.
	\end{align}
	In particular the above relation holds true for any $(\mu,p_1)\in\mathcal M_{\mathcal V}$. Letting $\delta_n\rightarrow 0$, we get
		\begin{align}\label{}
			\int_{{\mathcal P}(\T^d)}\chi_0(m)\mu(dm)\leq 0,\qquad\forall (\mu,p_1)\in\mathcal M_{\mathcal V},
		\end{align}
	which proves that $\chi_0\in{\mathcal S}^-$ and, consequently, that $\chi_0\leq\bar\chi$.
	
	For the other inequality we fix $m_0\in\mathcal P(\T^d)$ and a sequence $m_\varepsilon$ of smooth measures such that $m_\varepsilon\rightarrow m_0$ when $\varepsilon\rightarrow 0$. From Lemma \ref{punto.3} we know that, for any subsolution $\chi$ of \eqref{critical.HJ}, we have
	\begin{align}\label{ser.teo.fin.ine}
	\bar{\mathcal{V}}_{\delta_n}(m_\varepsilon)\geq\chi(m_\varepsilon)-\int_{\mathcal P(\T^d)}\chi(m)\nu_{\delta_n}^{m_\varepsilon}(dm).
	\end{align}
	As $m_\varepsilon$ is smooth, we know from Lemma \ref{lemma.nu.Mdelta.lim} that there exists a further subsequence $\delta_{n_k}$ such that $\nu_{\delta_{n_k}}^{m_\varepsilon}\rightarrow \nu^{m_\varepsilon}$ and $\nu^{m_\varepsilon}\in\mathcal M_{\mathcal V}$. Therefore, if we let in \eqref{ser.teo.fin.ine} ${\delta_{n_k}}\rightarrow 0$ , we get
	\begin{align}\label{}
	\chi_0(m_\varepsilon)\geq\chi(m_\varepsilon)-\int_{\mathcal P(\T^d)}\chi(m)\nu^{m_\varepsilon}(dm).
	\end{align}
	If now we suppose that $\chi\in{\mathcal S}^-$, the above inequality becomes
	\begin{align}\label{}
	\chi_0(m_\varepsilon)\geq\chi(m_\varepsilon).
	\end{align}
	As $\chi$ and $\chi_0$ are continuous we can let $\varepsilon\rightarrow 0$ to finally find that 
	\begin{align}\label{}
	\chi_0(m_0)\geq\chi(m_0).
	\end{align}
	By the arbitrariness of $\chi$ and $m_0$ we deduce that
	$$
	\chi_0\geq\sup_{\chi\in{\mathcal S}^-}\chi=\bar\chi.
	$$
	
	Note that $\bar\chi$ is uniquely defined and depends neither on $\delta_n$ nor on $\delta_{n_k}$. This implies that also $\chi_0$ is uniquely defined and, therefore, the full convergence of $\bar{\mathcal{V_\delta}}$.
\end{proof}

\bigskip
	\bibliography{KAMbib}
	\bibliographystyle{amsplain}
	
\end{document}